\documentclass{amsart}

\newtheorem{theorem}{Theorem}[section]
\newtheorem{lemma}[theorem]{Lemma}
\newtheorem{cor}[theorem]{Corollary}

\theoremstyle{definition}
\newtheorem{definition}[theorem]{Definition}

\theoremstyle{remark}

\numberwithin{equation}{section}

\begin{document}

\title[Anti-holomorphic semi-invariant submersions]{Anti-holomorphic semi-invariant\\
submersions from  K\"{a}hlerian manifolds}

\author{Hakan Mete Ta\c stan}

\address{\.Istanbul University\\
Department of Mathematics\\
Vezneciler, \.Istanbul, Turkey}

\email{hakmete@istanbul.edu.tr}


\begin{abstract}
We study anti-holomorphic semi-invariant submersions from
K\"{a}hlerian manifolds onto Riemannian manifolds. We prove that all
distributions which are involved in the definition of the submersion
are integrable. We also prove that the O'Neill's tensor
$\mathcal{T}$ vanishes on the invariant vertical distribution. We
give  necessary and sufficient conditions for totally geodesicness
and harmonicity of this type submersions. Moreover, we investigate
the several curvatures of the total manifold and fibers and give a
characterization theorem.
\end{abstract}
\maketitle
\section{Introduction}
The theory of Riemannian submersions were initiated by O'Neill
\cite{O} and  Gray \cite{Gr}. In \cite{Wat}, the Riemannian
submersions were considered between almost Hermitian manifolds by
Watson under the name of almost Hermitian submersions. In this case,
the Riemannian submersion is also an almost complex mapping and
consequently the vertical and horizontal distribution are invariant
with respect to the almost complex structure of the total manifold
of the submersion. Afterwards, almost Hermitian submersions have
been actively studied between different subclasses of almost
Hermitian manifolds, for example, see \cite{Fal}. Most of the
studies related to Riemannian or almost Hermitian submersions can be
found in the book \cite{Fa}. The study of anti-invariant Riemannian
submersions from almost Hermitian manifolds were initiated by
\d{S}ahin \cite{Sah}. In this case, the fibres are anti-invariant
with respect to the almost complex structure of the total manifold.
He studied this type submersions from a K\"{a}hlerian manifold onto
a Riemannian manifold. Recently, Shahid and Tanveer \cite{Sha}
extended this notion to the case when the total manifold is nearly
K\"{a}hlerian. A Lagrangian submersion is a special case of an
anti-invariant Riemannian submersion such that the almost complex
structure of the total manifold reverses the vertical and horizontal
distributions. In \cite{Ta}, we studied Lagrangian submersions in
detail. There are some other recent paper which involve other
structures such as almost product \cite{Gu}, almost contact
\cite{Le}, and Sasakian \cite{Km}. In any cases, the definition of
anti-invariant Riemannian submersion is the same as the above
definition. Besides there are many other notions related with that
of anti-invariant Riemannian submersion, such as slant submersion
\cite{Sa1}, semi-invariant submersion \cite{Sa} and semi-slant
submersion \cite{Pa}. In particular, the notion of semi-invariant is
a natural generalization of the notion anti-invariant submersion. In
this paper, we consider semi-invariant submersions from a
K\"{a}hlerian manifold onto a Riemannian manifold in a special case.

\section{Riemannian submersions}

In this section,  we give necessary background for
Riemannian submersions.\\

Let $(M,g)$ and $(N,g_{\text{\tiny$N$}})$ be Riemannian manifolds,
where $dim(M)>dim(N)$. A surjective mapping
$\pi:(M,g)\rightarrow(N,g_{N})$ is called a \emph{Riemannian
submersion}
\cite{O} if:\\

(S1) $\pi$ has maximal rank, and \\

(S2) $\pi_{*}$, restricted to $(ker\pi_{*})^{\bot},$ is a linear
isometry.\\

In this case, for each $y\in N$, $\pi^{-1}(y)$ is a $k$-dimensional
submanifold of $M$ and called \emph{fiber}, where $k=dim(M)-dim(N).$
A vector field on $M$ is called \emph{vertical} (resp.
\emph{horizontal}) if it is always tangent (resp. orthogonal) to
fibers. A vector field $X$ on $M$ is called \emph{basic} if $X$ is
horizontal and $\pi$-related to a vector field $X_{*}$ on $N,$ i.e.,
$\pi_{*}X_{x}=X_{*\pi(x)}$ for all $x\in M.$ As usual, we denote by
$\mathcal{V}$ and $\mathcal{H}$ the projections on the vertical
distribution $ker\pi_{*}$ and the horizontal distribution
$(ker\pi_{*})^{\bot},$ respectively. The geometry of Riemannian
submersions is characterized by O'Neill's tensors $\mathcal{T}$ and
$\mathcal{A}$, defined as follows:
\begin{equation}\label{testequation}
\mathcal{T}_{E}F=\mathcal{V}\nabla_{\mathcal{V}E}\mathcal{H}F+\mathcal{H}\nabla_{\mathcal{V}E}\mathcal{V}F,
\end{equation}
\begin{equation}\label{testequation}
\mathcal{A}_{E}F=\mathcal{V}\nabla_{\mathcal{H}E}\mathcal{H}F+\mathcal{H}\nabla_{\mathcal{H}E}\mathcal{V}F
\end{equation}
for any vector fields $E$ and $F$ on $M,$ where $\nabla$ is the
Levi-Civita connection of $g_{\text{\tiny$M$}}.$ It is easy to see
that $\mathcal{T}_{E}$ and $\mathcal{A}_{E}$ are skew-symmetric
operators on the tangent bundle of $M$ reversing the vertical and
the horizontal distributions. We summarize the properties of the
tensor fields $\mathcal{T}$ and $\mathcal{A}$. Let $U,V$ be
vertical and $\xi,\eta$ be horizontal vector fields on $M$, then we
have
\begin{equation}\label{testequation}
\mathcal{T}_{U}V=\mathcal{T}_{V}U,
\end{equation}
\begin{equation}\label{testequation}
\mathcal{A}_{\xi}\eta=-\mathcal{A}_{\eta}\xi=\frac{1}{2}\mathcal{V}[\xi,\eta].
\end{equation}
On the other hand, from (2.1) and (2.2), we obtain
\begin{equation}\label{testequation}
\nabla_{U}V=\mathcal{T}_{U}V+\hat{\nabla}_{U}V,
\end{equation}
\begin{equation}\label{testequation}
\nabla_{U}\xi=\mathcal{T}_{U}\xi+\mathcal{H}\nabla_{U}\xi,
\end{equation}
\begin{equation}\label{testequation}
\nabla_{\xi}U=\mathcal{A}_{\xi}U+\mathcal{V}\nabla_{\xi}U,
\end{equation}
\begin{equation}\label{testequation}
\nabla_{\xi}\eta=\mathcal{H}\nabla_{\xi}\eta+\mathcal{A}_{\xi}\eta,
\end{equation}
where $\hat{\nabla}_{U}V=\mathcal{V}\nabla_{U}V$ and
$\mathcal{H}\nabla_{V}\xi=\mathcal{A}_{\xi}V$, if $\xi$ is basic. It
is not difficult to observe that $\mathcal{T}$ acts on the fibers as
the second fundamental form while $\mathcal{A}$  acts on the
horizontal distribution and measures of the obstruction to the
integrability of this distribution. For  details on the Riemannian
submersions, we refer to O'Neill's paper \cite{O} and to
the book \cite{Fa}.\\

 Finally, we recall that the notion of the second fundamental form
of a map between Riemannian manifolds. Let $(M,g)$ and
$(N,g_{\text{\tiny$N$}})$ be Riemannian manifolds and
$\varphi:(M,g)\rightarrow(N,g_{\text{\tiny$N$}})$ a smooth map. Then
the second fundamental form of $\varphi$ is given by
$$\quad\quad\quad\quad(\nabla\varphi_{*})(E,F)=\nabla^{\varphi}_{E}\varphi_{*}F-\varphi_{*}(E,F)$$
for $E,F\in TM,$ where $\nabla^{\varphi}$ is the pull back
connection and we denoted conveniently by $\nabla$ the Riemannian
connections of the metrics $g$ and $g_{\text{\tiny$N$}}$ \cite{Ba}.

\section{Anti-holomorphic semi-invariant submersions}
A smooth manifold $M$ is called almost Hermitian \cite{Yan} if its
tangent bundle has an almost complex structure $J$ and a Riemannian
metric $g$ such that
\begin{equation}\label{testequation}
\quad g(E,F)=g(JE,JF)
\end{equation}
for any vector fields $E, F\in T{M}$, where $T{M}$ is the tangent
bundle of ${M}$. Before, giving our definition recall that the
definition of semi-invariant submersion.
\begin{definition} (\cite{Sa}) Let $M$ be a $2m$-dimensional almost Hermitian manifold with
Hermitian metric $g$ and almost complex structure $J$, and $N$ be a
Riemannian manifold with Riemannian metric $g_{\text{\tiny$N$}}.$ A
Riemannian submersion $\pi:(M,g,J)\rightarrow
(N,g_{\text{\tiny$N$}})$ is called a \emph{semi-invariant
submersion} if there is a distribution $\mathcal{D}\subset
ker\pi_{*}$ such that
\begin{equation} ker\pi_{*}=\mathcal{D}\oplus\mathcal{D}^{\bot},
\quad J(\mathcal{D})=\mathcal{D},\quad
J(\mathcal{D}^{\bot})\subset(ker\pi_{*})^{\bot},
\end{equation}
where $\mathcal{D}^{\bot}$ is the orthogonal complement of
$\mathcal{D}$ in $ker\pi_{*}$.
\end{definition}
In this case, the horizontal distribution $(ker\pi_{*})^{\bot}$ is
decomposed as
\begin{equation} (ker\pi_{*})^{\bot}=J(\mathcal{D}^{\bot})\oplus\mu,
\end{equation}
where $\mu$ is the orthogonal complementary distribution of
$J(\mathcal{D}^{\bot})$ in $(ker\pi_{*})^{\bot}$ and it is invariant
with respect to $J.$ Note that, a semi-invariant Riemannian
submersion is a natural generalization of an anti-invariant
Riemannian submersion \cite{Sah}. For the details, see \cite{Sah,
Sa}.

\begin{definition} Let $\pi:(M,g,J)\rightarrow(N,g_{\text{\tiny$N$}})$ be a semi-invariant submersion.
Then we call $\pi$ an \emph{anti-holomorphic semi-invariant
submersion}, if $(ker\pi_{*})^{\bot}=J(\mathcal{D}^{\bot})$, i.e.,
$\mu=\{0\}.$
\end{definition}
Suppose the dimension of distribution $\mathcal{D}^{\bot}$ (resp.
$\mathcal{D}^{\bot}$) is $2p$ (resp. $q$). Then, we have
$dim(M)=2p+2q$ and $dim(N)=q.$ An anti-holomorphic semi-invariant
submersion is called a \emph{proper
anti-holomorphic semi-invariant submersion} if $p\neq0$ and $q\neq0.$\\

\noindent {\bf Example.} Define
\quad$\pi:\mathbb{R}^{4}\rightarrow\mathbb{R}$\quad by
\quad$\pi(x_{1},x_{2},x_{3},x_{4})=\displaystyle\frac{x_{3}-x_{4}}{\sqrt{2}}.$\\

Then the map $\pi$ is a proper anti-holomorphic semi-invariant
submersion such
that\\

$ker\pi_{*}=\mathcal{D}\oplus\mathcal{D}^{\bot},$ where
$\mathcal{D}=span \{\partial_{1},\partial_{2}\}$,
$\mathcal{D}^{\bot}=span\{\partial_{3}+\partial_{4}\}$,\\

and $ker\pi_{*}^{\bot}=span \{\partial_{4}-\partial_{3}\}$,
$\partial_i={\partial\over \partial{x_i}}.$

\section{Anti-holomorphic semi-invariant\\
submersions from K\"{a}hlerian manifolds} In this section, we start
to study anti-holomorphic semi-invariant submersions from
K\"{a}hlerian manifolds. An almost Hermitian manifold $(M,g,J)$ is
called a\emph{ K\"{a}hlerian manifold } if
\begin{equation}\label{testequation}
(\nabla_{E}J)F=0
\end{equation}
for all $E,F\in TM.$ Let $(M,g,J)$ be a K\"{a}hlerian manifold and
$(N,g_{\text{\tiny$N$}})$ be a Riemannian manifold. Now we examine
how the K\"{a}hlerian structure on $M$ places restrictions on the
tensor fields $\mathcal{T}$ and $\mathcal{A}$ of an anti-holomorphic
semi-invariant submersion $\pi:(M,g,J)\rightarrow
(N,g_{\text{\tiny$N$}})$. In \cite{Ta}, we proved that the following
lemma for Lagrangian submersions. For the details of Lagrangian
submersions, see \cite{Sah,Ta}.
\begin{lemma}(\cite{Ta})
Let $\pi$ be a Lagrangian submersion from a K\"{a}hlerian manifold
$(M,g,J)$ onto a Riemannian manifold $(N,g_{\text{\tiny$N$}})$. Then
we have
$$\textbf{a)}\quad\mathcal{T}_{V}JE=J\mathrm{T}_{V}E\quad\quad\quad\quad\textbf{b)}\quad\mathcal{A}_{\xi}JE=J\mathcal{A}_{\xi}E,$$
where $V$ is a vertical vector field, $\xi$ is a horizontal vector
field, and $E$ is a vector field on $M.$
\end{lemma}
It is easy to show that this lemma also holds for anti-holomorphic
semi-invariant submersions.

\section{Integrability and Totally Geodesicness }

In this section, we shall study the integrability and totally
geodesicness of the distributions which are involved in the
definition of an anti-holomorphic semi-invariant submersion.\\

In \cite{Sa}, \c{S}ahin proved that the following.
\begin{lemma}(\cite{Sa})
Let $\pi$ be a semi-invariant submersion from a K\"{a}hlerian
manifold $(M,g,J)$ onto a Riemannian manifold
$(N,g_{\text{\tiny$N$}})$. Then

\textbf{a)} The anti-invariant distribution $\mathcal{D}^{\bot}$ is
always integrable.

\textbf{b)} The invariant distribution $\mathcal{D}$ is integrable
if and only if

\quad\quad\quad\quad\quad\quad\quad\quad$g(\mathcal{T}_{Z}JW-\mathcal{T}_{W}JZ,JX)=0$

for $Z,W\in\mathcal{D}$ and $X\in\mathcal{D}^{\bot}.$
\end{lemma}
Thus, using Lemma 4.1 and (2.3), we easily conclude that the
following result from Lemma 5.1.
\begin{lemma} Let $\pi$ be an anti-holomorphic semi-invariant submersion from a K\"{a}hlerian
manifold $(M,g,J)$ onto a Riemannian manifold
$(N,g_{\text{\tiny$N$}})$. Then

\textbf{a)} The anti-invariant distribution $\mathcal{D}^{\bot}$ is
always integrable.

\textbf{b)} The invariant distribution $\mathcal{D}$ is always
integrable.
\end{lemma}
Now, we state one of the main results.
\begin{theorem} Let $\pi$ be an anti-holomorphic semi-invariant submersion from a K\"{a}hlerian
manifold $(M,g,J)$ onto a Riemannian manifold
$(N,g_{\text{\tiny$N$}})$. Then horizontal distribution
$(ker\pi_{*})^{\bot}$ is integrable and totally geodesic, i.e.,
$\mathcal{A}\equiv0.$
\end{theorem}
\begin{proof} It is very similar to the proof of Theorem 4.5(\cite{Ta}), so we omit it.
\end{proof}
We remark that the vertical distribution $ker\pi_{*}$ is always
integrable.
\begin{lemma} Let $\pi$ be an anti-holomorphic semi-invariant submersion from a K\"{a}hlerian
manifold $(M,g,J)$ onto a Riemannian manifold
$(N,g_{\text{\tiny$N$}})$. Then the anti-invariant distribution
$\mathcal{D}^{\bot}$ defines a totally geodesic foliation in the
fibers $\pi^{-1}(y), y\in N.$
\end{lemma}
\begin{proof} Let $X,Y\in\mathcal{D}^{\bot}$ and $Z\in\mathcal{D}.$ Then using (2.5) and Lemma 4.1, we get
$g(\hat{\nabla}_{X}Y,Z)=g(\nabla_{X}Y,Z)=g(-J\nabla_{X}JY,Z)=g(\nabla_{X}JY,JZ)=g(\mathcal{T}_{X}JY,JZ)\\
\quad\quad\quad=g(J\mathcal{T}_{X}Y,JZ)=g(\mathcal{T}_{X}Y,Z)=0.$
This completes the proof.
\end{proof}

In a similar way, we have the following result.
\begin{lemma} Let $\pi$ be an anti-holomorphic semi-invariant submersion from a K\"{a}hlerian
manifold $(M,g,J)$ onto a Riemannian manifold
$(N,g_{\text{\tiny$N$}})$. Then the invariant distribution
$\mathcal{D}$ defines a totally geodesic foliation in the fibers
$\pi^{-1}(y), y\in N.$
\end{lemma}

By Lemma 5.4 and 5.5, we have that:
\begin{theorem} Let $\pi$ be an anti-holomorphic semi-invariant submersion from a K\"{a}hlerian
manifold $(M,g,J)$ onto a Riemannian manifold
$(N,g_{\text{\tiny$N$}})$. The the fibers of $\pi$ are locally
product Riemannian manifolds.
\end{theorem}

Now, we look more closely at the O' Neill's tensor $\mathcal{T}$ of
the anti-holomorphic semi-invariant submersion $\pi.$ Let $U,V\in
ker\pi_{*}$ and $\xi\in (ker\pi_{*})^{\bot}.$ Since
$(ker\pi_{*})^{\bot}=J(\mathcal{D}^{\bot})$, there is a vertical
vector field $X\in\mathcal{D}^{\bot}$ such that $\xi=JX.$ Then, we
have
$g(\mathcal{T}_{U}V,\xi)=g(\mathcal{T}_{U}V,JX)=-g(J\mathcal{T}_{U}V,X)=-g(\mathcal{T}_{U}JV,X).$
Hence for any $V\in\mathcal{D}$, we get
\begin{equation}\label{testequation}
g(\mathcal{T}_{U}V,\xi)=0.
\end{equation}
From (5.1), we deduce that
\begin{equation}\label{testequation}
\mathcal{T}_{U}\mathcal{D}=0
\end{equation}
for any $U\in ker\pi_{*}.$\\

Thus, using last equation (5.2), we have the following our main
result.
\begin{theorem} Let $\pi$ be an anti-holomorphic semi-invariant submersion from a K\"{a}hlerian
manifold $(M,g,J)$ onto a Riemannian manifold
$(N,g_{\text{\tiny$N$}})$. Then, we have always\\

\quad\quad \quad\textbf{a)}
\quad$\mathcal{T}_{X}Z=0=\mathcal{T}_{Z}X$\quad\quad\quad\textbf{b)}
\quad$\mathcal{T}_{Z}W=0,$\\

where $X\in \mathcal{D}^{\bot}$ and $Z,W\in\mathcal{D}.$
\end{theorem}

At once, from  Theorem 5.7, we easily see that
$\quad\mathcal{T}_{Z}\xi=0$ for any $Z\in\mathcal{D}$ and $\xi\in
(ker\pi_{*})^{\bot}.$ Thus, we have
\begin{cor}
Let $\pi$ be an anti-holomorphic semi-invariant submersion from a
K\"{a}hlerian manifold $(M,g,J)$ onto a Riemannian manifold
$(N,g_{\text{\tiny$N$}})$. Then, we have always
$\mathcal{T}_{Z}\equiv0$ for $Z\in\mathcal{D}$.
\end{cor}
From the part \textbf{a)} of Theorem 5.7, we have that:
\begin{cor}
Let $\pi$ be an anti-holomorphic semi-invariant submersion from a
K\"{a}hlerian manifold $(M,g,J)$ onto a Riemannian manifold
$(N,g_{\text{\tiny$N$}})$. Then the fibers of $\pi$ are always mixed
totally geodesic.
\end{cor}
From the part \textbf{b)} of Theorem 5.7, we get:
\begin{cor}
Let $\pi$ be an anti-holomorphic semi-invariant submersion from a
K\"{a}hlerian manifold $(M,g,J)$ onto a Riemannian manifold
$(N,g_{\text{\tiny$N$}})$. Then the foliations of the invariant
distribution $\mathcal{D}$ are totally geodesic in the total space
$M.$
\end{cor}
Also from Theorem 5.7, it follows that:
\begin{cor}
Let $\pi$ be an anti-holomorphic semi-invariant submersion from a
K\"{a}hlerian manifold $(M,g,J)$ onto a Riemannian manifold
$(N,g_{\text{\tiny$N$}})$. Then $\mathcal{T}\equiv0$ if and only if
$\mathcal{T}_{X}Y=0$ for all $X,Y\in \mathcal{D}^{\bot}$, i.e.,
$\mathcal{T}_{\mathcal{D}^{\bot}}\mathcal{D}^{\bot}=0.$
\end{cor}
Thus, we obtain the following result.
\begin{cor}
Let $\pi$ be an anti-holomorphic semi-invariant submersion from a
K\"{a}hlerian manifold $(M,g,J)$ onto a Riemannian manifold
$(N,g_{\text{\tiny$N$}})$. Then $ker\pi_{*}$ defines a totally
geodesic foliation if and only if
\quad$\mathcal{T}_{\mathcal{D}^{\bot}}\mathcal{D}^{\bot}=0.$
\end{cor}
Since the O'Neill's tensor $\mathcal{A}\equiv0,$ by Corollary 5.12,
we have the following.
\begin{theorem} Let $\pi$ be an anti-holomorphic semi-invariant submersion from a K\"{a}hlerian
manifold $(M,g,J)$ onto a Riemannian manifold
$(N,g_{\text{\tiny$N$}})$. Then, $M$ is a locally product Riemannian
manifold $M_{ker\pi_{*}}\times M_{(ker\pi_{*})^{\bot}}$ if and only
if \quad$\mathcal{T}_{\mathcal{D}^{\bot}}\mathcal{D}^{\bot}=0.$
\end{theorem}

\section{Totally Geodesicness and Harmonicity of the\\
anti-holomorphic semi-invariant submersion}

In this section, we shall examine the totally geodesicness and
harmonicity of an anti-holomorphic semi-invariant submersion. First
we give a necessary and sufficient condition for an anti-holomorphic
semi-invariant submersion to be a totally geodesic map. Recall that
a smooth map $\varphi$ between two Riemannian manifolds is called
totally geodesic if $\nabla\varphi_{*}=0$ \cite{Ba}.

\begin{theorem} Let $\pi$  be an anti-holomorphic semi-invariant submersion from a K\"{a}hlerian
manifold $(M,g,J)$ onto a Riemannian manifold
$(N,g_{\text{\tiny$N$}})$. Then $\pi$ is a totally geodesic map if
and only if\quad
$\mathcal{T}_{\mathcal{D}^{\bot}}\mathcal{D}^{\bot}=0.$
\end{theorem}
\begin{proof} Since $\pi$ is a Riemannian submersion, we have
\begin{equation}\label{testequation}
(\nabla\pi_{*})(\xi,\eta)=0
\end{equation}
for all $\xi,\eta\in(ker\pi_{*})^{\bot}.$ For any $U,V\in
ker\pi_{*}$, using (2.5), we get
$(\nabla\pi_{*})(U,V)=-\pi_{*}(\nabla_{U}V)=-\pi_{*}(\mathcal{T}_{U}V+\hat{\nabla}_{U}V)=-\pi_{*}(\mathcal{T}_{U}V),$
since $\pi$ is a linear isometry between $(ker\pi_{*})^{\bot}$ and
$TN.$ Hence, it follows that $(\nabla\pi_{*})(U,V)=0$ if and only if
$\mathcal{T}_{U}V=0,$ for all $U,V\in ker\pi_{*}$, that is;
\begin{equation}\label{testequation}
(\nabla\pi_{*})(U,V)=0\Leftrightarrow\mathcal{T}\equiv0.
\end{equation}
In a similar way, for any $U\in ker\pi_{*}$ and
$\xi\in(ker\pi_{*})^{\bot}$, using (2.7), we get
$(\nabla\pi_{*})(\xi,U)=-\pi_{*}(\nabla_{\xi}U)=-\pi_{*}(\mathcal{A}_{\xi}V+\mathcal{V}{\nabla}_{\xi}U).$
Since $\pi$ is a linear isometry between $(ker\pi_{*})^{\bot}$ and
$TN$ and $\mathcal{A}\equiv0,$ it follows that
\begin{equation}\label{testequation}
(\nabla\pi_{*})(\xi,U)=0
\end{equation}
for any $U\in ker\pi_{*}$ and $\xi\in(ker\pi_{*})^{\bot}.$ Thus,
from (6.1), (6.2) and (6.3), we deduce $\nabla\pi_{*}=0$ if and only
if $\mathcal{T}\equiv0.$ But because of Corollary 5.11, this is
equivalent to the assertion.
\end{proof}
Now, we examine the harmonicity of the submersion. We know that a
smooth map
 $\varphi$ is harmonic if and only if it has minimal fibers \cite{Ba}. Thus
 the submersion $\pi$ is harmonic if and only if $\displaystyle\sum^{2p+q}_{k=1}\mathcal{T}_{v_{k}}v_{k}=0,$
where $\{v_{1},...,v_{2p+q}\}$ is a local orthonormal frame of
$ker\pi_{*}.$ But because of Theorem 5.7, it follows that $\pi$ is
harmonic if and only if
$\displaystyle\sum^{q}_{i=1}\mathcal{T}_{e_{i}}e_{i}=0,$ where
$\{e_{1},...,e_{q}\}$ is a local orthonormal frame of
$\mathcal{D}^{\bot}.$ Next, let $X$  be any non-zero vector field in
$\mathcal{D}^{\bot}.$ Then, for $1\leq i\leq q$, using the
skew-symmetricness of $\mathcal{T}_{E}$, Lemma 4.1 and (2.3), we
have
$g(\mathcal{T}_{e_{i}}e_{i},JX)=-g(\mathcal{T}_{e_{i}}JX,e_{i})=-g(J\mathcal{T}_{e_{i}}X,e_{i})=-g(J\mathcal{T}_{X}e_{i},e_{i}).$
Hence, we get
\begin{equation}\label{testequation}
g(\displaystyle\sum^{q}_{i=1}\mathcal{T}_{e_{i}}e_{i},JX)=-\displaystyle\sum^{q}_{i=1}g(J\mathcal{T}_{X}e_{i},e_{i}).
\end{equation}
for all $X\in\mathcal{D}^{\bot}.$ Thus, from (6.4), we have the
following result.
\begin{theorem} Let $\pi$ be an anti-holomorphic semi-invariant submersion from a K\"{a}hlerian
manifold $(M,g,J)$ onto a Riemannian manifold
$(N,g_{\text{\tiny$N$}})$. Then $\pi$ is harmonic if and only if
\quad$traceJ\mathcal{T}_{X}=0$ for all $X\in\mathcal{D}^{\bot}.$
\end{theorem}
\section{The Geometry of Total Manifold and Fibers}
Lastly, we investigate several curvatures of the total manifold and
fibers and give a characterization theorem for this type submersions.
First, we recall that fundamental definitions and notions. Let
$(M,g,J)$ be a K\"{a}hlerian manifold and  $\nabla$ is the
Levi-Civita connection on $M.$ The Riemannian curvature tensor
\cite{Yan} of $(M,g,J)$ is defined by
$\mathrm{R}(E,F)G=\nabla_{[E,F]}G-[\nabla_{E},\nabla_{F}]G$ for
vector fields $E,F$ and $G$ on $M$. We put
$\mathrm{R}(E,F;G,\bar{G})=g(\mathrm{R}(E,F)G,\bar{G})$ where
$\bar{G}$ is a vector field on $M$. The \emph{sectional curvature}
$\mathrm{K}(E,F)$ of the plane $\sigma$ spanned by the orthogonal
unit vector fields $E$ and  $F$, is defined by
\begin{equation}\label{testequation}
\mathrm{K}(E,F)=\mathrm{R}(E,F;E,F).
\end{equation}

The \emph{holomorphic bisectional curvature} \cite{Go} of $M$ is
defined for any pair unit  vector fields $E$ and  $F$ tangent to $M$
by
\begin{equation}\label{testequation}
\mathrm{B}(E,F)=\mathrm{R}(E,JE;F,JF).
\end{equation}
Then the \emph{holomorphic sectional curvature} \cite{Go,Yan} of $M$
is given by
\begin{equation}\label{testequation}
\mathrm{H}(E)=\mathrm{B}(E,E).
\end{equation}
The manifold $M$ is called a \emph{complex space form} if it is of
constant holomorphic sectional curvature. We denote by $(M,g,J)(c)$
a complex space form of constant holomorphic sectional curvature
$c.$ Then the Riemannian curvature tensor $\mathrm{R}$ of
$(M,g,J)(c)$ is given by
\begin{equation}
\begin{array}{c}
\mathrm{R}(E,F)G=\frac{c}{4}\{g(F,G)E-g(E,G)F+g(JF,G)JE \\
\quad\quad\quad\quad\quad\quad\quad\quad\quad\quad\quad\quad\quad-g(JE,G)JF+2g(E,JF)JG\}
\end{array}
\end{equation}
for any vector fields  $E,F$ and $G$ on $M$. Hence, we have
\begin{equation}
\begin{array}{c}
\mathrm{B}(E,F)=\frac{c}{2}\{g(E,E)g(F,F)+(g(E,F))^{2}+(g(E,JF))^{2}\}
\end{array}.
\end{equation}
We note that a K\"{a}hlerian manifold with vanishing holomorphic
sectional curvature is flat \cite{Go,Yan}.\\

In view of the O'Neill's curvature formulas $\{0\}$, $\{1\}$,
$\{2\}$, $\{2'\}$ \cite{O}, Lemma 4.1, Theorem 5.7 and Corollary
5.8, from (7.1), we get the following curvature formulas.

\begin{theorem} Let $\pi$  be an anti-holomorphic semi-invariant submersion from a K\"{a}hlerian
manifold $(M,g,J)$ onto a Riemannian manifold
$(N,g_{\text{\tiny$N$}})$ and let $K, \hat{K}$ and $K_{*}$ be the
sectional curvatures of the total space $M$, fibers and the base
space $N$, respectively. Then
\begin{equation}\label{testequation}
\mathrm{K}(X,Y)=\hat{\mathrm{K}}(X,Y)-g(\mathcal{T}_{X}X,\mathcal{T}_{Y}Y)+\|\mathcal{T}_{X}Y\|^{2},
\end{equation}
\begin{equation}\label{testequation}
\mathrm{K}(X,Z)=\hat{\mathrm{K}}(X,Z),\quad
\end{equation}
\begin{equation}\label{testequation}
\mathrm{K}(Z,W)=\hat{\mathrm{K}}(Z,W),
\end{equation}
\begin{equation}\label{testequation}
\mathrm{K}(X,\xi)=g((\nabla_{\xi}\mathcal{T})_{X}X,\xi)-\|\mathcal{T}_{X}\xi\|^{2},
\end{equation}
\begin{equation}\label{testequation}
\mathrm{K}(Z,\xi)=g((\nabla_{\xi}\mathcal{T})_{Z}Z,\xi)
\end{equation}
\begin{equation}\label{testequation}
\mathrm{K}(\xi,\eta)=\mathrm{K}_{*}(\xi_{*},\eta_{*}),
\end{equation}
where $X,Y\in \mathcal{D}^{\bot},$  $Z,W\in \mathcal{D},$  $\xi,
\eta\in (ker\pi_{*})^{\bot}$, $\xi_{*}=\pi_{*}(\xi),
\eta_{*}=\pi_{*}(\eta)$, and all of them are unit vector fields.
\end{theorem}
From (7.6), (7.7) and (7.8), we have the following result.
\begin{cor}
Let $\pi$ be an anti-holomorphic semi-invariant submersion from a
K\"{a}hlerian manifold $(M,g,J)$ onto a Riemannian manifold
$(N,g_{\text{\tiny$N$}})$. Then any fiber of $\pi^{-1}(y)$ of $\pi$
has constant sectional curvature if and only if
$g(\mathcal{T}_{X}X,\mathcal{T}_{Y}Y)=\|\mathcal{T}_{X}Y\|^{2}$ for
all $X,Y\in \mathcal{D}^{\bot}.$
\end{cor}

\begin{theorem} Let $\pi$  be an anti-holomorphic semi-invariant submersion from a K\"{a}hlerian
manifold $(M,g,J)$ onto a Riemannian manifold
$(N,g_{\text{\tiny$N$}})$ and let $B$ and $\hat{B}$ be the
holomorphic bisectional curvatures of the total space $M$ and
fibers, respectively. Then
\begin{equation}\label{testequation}
\mathrm{B}(X,Y)=g((\nabla_{JX}\mathcal{T})_{X}Y,JY)-g(\mathcal{T}_{X}X,\mathcal{T}_{Y}Y),
\end{equation}
\begin{equation}\label{testequation}
\mathrm{B}(X,Z)=0,
\end{equation}
\begin{equation}\label{testequation}
\mathrm{B}(Z,W)=\hat{\mathrm{B}}(Z,W),
\end{equation}
\begin{equation}\label{testequation}
\mathrm{B}(X,\xi)=-g((\nabla_{JX}\mathcal{T})_{X}J\xi,\xi)+g(\mathcal{T}_{X}JX,\mathcal{T}_{J\xi}\xi),
\end{equation}
\begin{equation}\label{testequation}
\mathrm{B}(Z,\xi)=g((\nabla_{Z}\mathcal{T})_{JZ}J\xi,\xi)-g((\nabla_{JZ}\mathcal{T})_{Z}J\xi,\xi),
\end{equation}
\begin{equation}\label{testequation}
\mathrm{B}(\xi,\eta)=g((\nabla_{\xi}\mathcal{T})_{J\xi}J\eta,\eta)-g(\mathcal{T}_{J\xi}\xi,\mathcal{T}_{J\eta}\eta),
\end{equation}
where $X,Y\in \mathcal{D}^{\bot},$  $Z,W\in \mathcal{D},$  $\xi,
\eta\in (ker\pi_{*})^{\bot}$,
 and all of them are unit vector fields.
\end{theorem}

\begin{proof} From (7.2), using similar arguments which used in Theorem 7.1,  we
get all curvature formulas above expect (7.13). Next, we prove
(7.13). Using the O' Neill's formula $\{1\}$ \cite{O}, we have
\begin{equation}\label{testequation}
B(X,Z)=g((\nabla_{JZ}\mathcal{T})_{Z}X,JX)-g((\nabla_{Z}\mathcal{T})_{JZ}X,JX)
\end{equation}
for unit vector fields $X\in \mathcal{D}^{\bot}$ and  $Z\in
\mathcal{D}.$ After some calculation, from (7.18), we get
\begin{equation}\label{testequation}
B(X,Z)=g(\mathcal{T}_{\hat{\nabla}_{Z}JZ}X-\mathcal{T}_{\hat{\nabla}_{JZ}Z}X,JX).
\end{equation}
Because of Lemma 5.5, we know that $\hat{\nabla}_{Z}JZ,
\hat{\nabla}_{JZ}Z\in\mathcal{D}.$ Hence, by Theorem 5.7, we find
$\mathcal{T}_{\hat{\nabla}_{Z}JZ}X=\mathcal{T}_{\hat{\nabla}_{JZ}Z}X=0.$
Thus, (7.19) gives $\mathrm{B}(X,Z)=0.$
\end{proof}

We have immediately from Theorem 7.3 that:

\begin{cor} Let $\pi$  be an anti-holomorphic semi-invariant submersion from a K\"{a}hlerian
manifold $(M,g,J)$ onto a Riemannian manifold
$(N,g_{\text{\tiny$N$}})$ and let $\mathrm{H}$ and
$\hat{\mathrm{H}}$ be the holomorphic sectional curvatures of the
total space $M$ and fibers, respectively. Then
\begin{equation}\label{testequation}
\mathrm{H}(X)=g((\nabla_{JX}\mathcal{T})_{X}X,JX)-\|\mathcal{T}_{X}X\|^{2}
\end{equation}
\begin{equation}\label{testequation}
\mathrm{H}(Z)=\hat{\mathrm{H}}(Z),
\end{equation}
\begin{equation}\label{testequation}
\mathrm{H}(\xi)=g((\nabla_{\xi}\mathcal{T})_{J\xi}J\xi,\xi)-\|\mathcal{T}_{J\xi}\xi\|^{2}
\end{equation}
where $X\in \mathcal{D}^{\bot},$  $Z\in \mathcal{D}$ and $\xi\in
(ker\pi_{*})^{\bot}$ and all of them are unit vector fields.
\end{cor}

 With the help of (7.4) and (7.5), from  (7.13), we have the following result.

\begin{theorem} Let $\pi$  be a proper anti-holomorphic semi-invariant submersion from a complex space form
$(M,g,J)(c)$ onto a Riemannian manifold $(N,g_{\text{\tiny$N$}})$,
then $c=0.$ In other word, the total space is flat. In particular, there exists no proper
anti-holomorphic semi-invariant submersion from a complex space form
$(M,g,J)(c)$ with $c\neq0.$
\end{theorem}

From Theorem 7.5, we deduce that:

\begin{theorem} Let $\pi$ be an anti-holomorphic semi-invariant submersion from a complex space form
$(M,g,J)(c)$ with $c\neq0$ onto a Riemannian manifold
$(N,g_{\text{\tiny$N$}})$, then $\pi$ is either an anti-invariant
submersion (Lagrangian case) or an almost Hermitian submersion
(K\"{a}hlerian case).
\end{theorem}

\bibliographystyle{amsplain}

\end{document}